\documentclass[10pt]{amsart}
\usepackage[leqno]{amsmath}
\usepackage{amssymb,latexsym,soul,cite,amsthm,color,enumitem,graphicx,tikz,mathtools,microtype,accents}
\usepackage[colorlinks=true,urlcolor=ao(english),citecolor=ao(english),linkcolor=ao(english),linktocpage,pdfpagelabels,bookmarksnumbered,bookmarksopen]{hyperref}
\definecolor{ao(english)}{rgb}{0.0, 0.5, 0.0}
\usepackage[english]{babel}
\usepackage[left=2.5cm,right=2.5cm,top=2.5cm,bottom=2.5cm]{geometry}

\numberwithin{equation}{section}

\newtheorem{theorem}{Theorem}[section]
\theoremstyle{plain}
\newtheorem{lemma}[theorem]{Lemma}
\theoremstyle{plain}

\theoremstyle{plain}

\newtheorem{definition}[theorem]{Definition}
\theoremstyle{definition}
\newtheorem{remark}[theorem]{Remark}

\newcommand{\N}{{\mathbb N}}

\newcommand{\R}{{\mathbb R}}
\newcommand{\eps}{\varepsilon}
\newcommand{\beq}{\begin{equation}}
\newcommand{\eeq}{\end{equation}}
\renewcommand{\le}{\leqslant}
\renewcommand{\ge}{\geqslant}

\newcommand{\w}{W^{s,p}_0(\Omega)}
\newcommand{\fpl}{(-\Delta)_p^s\,}
\newcommand{\ds}{{\rm d}_\Omega^s}

\makeatletter
\newcommand{\leqnomode}{\tagsleft@true}
\newcommand{\reqnomode}{\tagsleft@false}
\makeatother

\newenvironment{enumroman}{\begin{enumerate}

}{\end{enumerate}}

\title[Extremal solutions for the fractional $p$-Laplacian]{Extremal constant sign solutions and nodal solutions for the fractional $p$-Laplacian}

\author[S.\ Frassu, A.\ Iannizzotto]{Silvia Frassu, Antonio Iannizzotto}

\address[S.\ Frassu]{Department of Mathematics and Computer Science
\newline\indent
University of Cagliari
\newline\indent
Viale L. Merello 92, 09123 Cagliari, Italy}
\email{silvia.frassu@unica.it}

\address[A.\ Iannizzotto]{Department of Mathematics and Computer Science
\newline\indent
University of Cagliari
\newline\indent
Viale L. Merello 92, 09123 Cagliari, Italy}
\email{antonio.iannizzotto@unica.it}


\begin{document}

\begin{abstract}
We study a pseudo-differential equation driven by the degenerate fractional $p$-Laplacian, under Dirichlet type conditions in a smooth domain. First we show that the solution set within the order interval given by a sub-supersolution pair is nonempty, directed, and compact, hence endowed with extremal elements. Then, we prove existence of a smallest positive, a biggest negative and a nodal solution, combining variational methods with truncation techniques.
\end{abstract}

\maketitle

\begin{center}
Version of \today\
\end{center}

\section{Introduction}\label{sec1}


\noindent
In the study of nonlinear boundary value problem, one classical issue is that about the sign of solutions, especially in the case of multiple solutions. Typically, constant sign solutions can be detected as critical points of a truncated energy functional by direct minimization or min-max methods, while the existence of a {\em nodal} (i.e., sign-changing) solution is a more delicate question (some classical results, based on Morse theory, can be found in \cite{AL,BL,ZL}). An interesting approach was proposed in \cite{DD} for the Dirichlet problem driven by the Laplacian operator: it consists in proving that the problem admits a {\em smallest} positive and a {\em biggest} negative solution, plus a third nontrivial solution lying between the two, which must then be nodal. The method used for finding the nodal solution is based on the Fu\v{c}ik spectrum. Such approach was then extended to the $p$-Laplacian in \cite{CP}, and then combined with a variational characterization of the second eigenvalue to detect a nodal solution under more general assumptions in \cite{FP} (see also \cite{GP,MMP1} and the monograph \cite{MMP}).
\vskip2pt
\noindent
Recently, many authors have devoted their attention to nonlinear equations driven by nonlocal operators. The present paper is devoted to the study of the following Dirichlet-type problem for a nonlinear fractional equation:
\beq\label{p}
\begin{cases}
\fpl u = f(x,u)& \text{in $\Omega$} \\
u=0 & \text{in $\Omega^c$,}
\end{cases}
\eeq
where $\Omega\subset\R^N$ ($N>1$) is a bounded domain with $C^{1,1}$ boundary, $p \ge 2$, $s\in(0,1)$, $N>ps$, and $\fpl$ denotes the fractional $p$-Laplacian, namely the nonlinear, nonlocal operator defined for all $u:\R^N\to\R$ smooth enough and all $x\in\R^N$ by
\beq\label{fpl}
\fpl u(x)=2\lim_{\eps\to 0^+}\int_{B_\eps^c(x)}\frac{|u(x)-u(y)|^{p-2}(u(x)-u(y))}{|x-y|^{N+ps}}\,dy
\eeq
(which in the linear case $p=2$ reduces to the fractional Laplacian up to a dimensional constant $C(N,p,s)>0$, see \cite{CS, CS1, DPV}). The reaction $f:\Omega\times\R\to\R$ is a Carath\'eodory mapping subject to a subcritical growth condition.
\vskip2pt
\noindent
Problem \eqref{p} has been intensively studied in the recent literature, both in the semilinear and the nonlinear case. Regarding the semilinear case, we recall the fine regularity results of \cite{RS}, the existence and multiplicity results obtained for instance in \cite{DI,FP2,IMS,U}, and the study on extremal solutions in \cite{RS1} (see also the monograph \cite{MRS}). The nonlinear case is obviously more involved: spectral properties of $\fpl$ were studied in \cite{BP,FP1,FRS,HS,LL}, a detailed regularity theory was developed in \cite{BL1,IMS1,IMS2,KMS,KMS1} (some results about Sobolev and H\"older regularity being only proved for the degenerate case $p>2$), maximum and comparison principles have appeared in \cite{DQ,J}, while existence and multiplicity of solutions have been obtained for instance in \cite{CMS,DQ1,FRS,ILPS,XZR} (see also the surveys \cite{MS,P}). For the purposes of the present study, we recall in particular \cite{IMS3}, where it was proved that the local minimizers of the energy functional corresponding to problem \eqref{p} in the topologies of $\w$ and of the weighted H\"older space $C^0_s(\overline\Omega)$, respectively, coincide (namely, a nonlinear fractional analogue of the classical result of \cite{BN}).
\vskip2pt
\noindent
Here we focus on the structure of the set $\mathcal{S}(\underline{u},\overline{u})$, namely the set of solutions of \eqref{p} lying within the interval $[\underline{u},\overline{u}]$ where $\underline{u}$ and $\overline{u}$ are a subsolution and a supersolution of \eqref{p}, respectively, with $\underline{u}\le\overline{u}$ in $\Omega$. We shall prove that $\mathcal{S}(\underline{u},\overline{u})$ is nonempty, directed, and compact in $\w$, hence endowed with extremal elements.
\vskip2pt
\noindent
Then, we will assume that $f(x,\cdot)$ is $(p-1)$-sublinear at infinity and asymptotically linear near the origin without resonance on the first eigenvalue, and prove that \eqref{p} has a smallest positive solution $u_+$ and a biggest negative solution $u_-$. Finally, under more restrictive assumptions on the behavior of $f(x,\cdot)$ near the origin, we will prove existence of a nodal solution $\tilde u$ s.t.\ $u_-\le\tilde u\le u_+$ in $\Omega$, thus extending some results of \cite{CP,FP} to the fractional $p$-Laplacian.
\vskip2pt
\noindent
We remark that our results are new (to our knowledge) even in the semilinear case $p=2$, and that the structure of the set $\mathcal{S}(\underline{u},\overline{u})$ can provide valuable information about extremal solutions also in different frameworks.
\vskip2pt
\noindent
The paper has the following structure: in Section \ref{sec2} we collect the necessary preliminaries; in Section \ref{sec3} we study the properties of the solution set; in Section \ref{sec4} we show existence of extremal constant sign solutions; and in Section \ref{sec5} we prove existence of a nontrivial nodal solution.
\vskip4pt
\noindent
{\bf Notation:} Throughout the paper, for any $A\subset\R^N$ we shall set $A^c=\R^N\setminus A$. For any two measurable functions $f,g:\Omega\to\R$, $f\le g$ will mean that $f(x)\le g(x)$ for a.e.\ $x\in\Omega$ (and similar expressions). The positive (resp., negative) part of $f$ is denoted $f^+$ (resp., $f^-$). If $X$ is an ordered Banach space, then $X_+$ will denote its non-negative order cone. For all $r\in[1,\infty]$, $\|\cdot\|_r$ denotes the standard norm of $L^r(\Omega)$ (or $L^r(\R^N)$, which will be clear from the context). Every function $u$ defined in $\Omega$ will be identified with its $0$-extension to $\R^N$. Moreover, $C$ will denote a positive constant (whose value may change case by case).

\section{Preliminaries}\label{sec2}

\noindent
In this section we collect some useful results related to the fractional $p$-Laplacian. First we fix a functional-analytical framework, following \cite{DPV, ILPS}. First, for all measurable $u:\R^N\to\R$ we set
\[[u]_{s,p}^p=\iint_{\R^N\times\R^N}\frac{|u(x)-u(y)|^p}{|x-y|^{N+ps}}\,d\mu,\]
where $d\mu=|x-y|^{-N-ps}\,dx\,dy$. Then we define the following fractional Sobolev spaces:
\[W^{s,p}(\R^N)=\big\{u\in L^p(\R^N):\,[u]_{s,p}<\infty\big\},\]
\[\w=\big\{u\in W^{s,p}(\R^N):\,u(x)=0 \ \text{in $\Omega^c$}\big\},\]
the latter being a uniformly convex, separable Banach space with norm $\|u\|_{s,p}=[u]_{s,p}$ and dual $W^{-s,p'}(\Omega)$ (with norm $\|\cdot\|_{-s,p'})$. Set $p_s^*= Np / (N-ps)$, then the embedding $\w\hookrightarrow L^q(\Omega)$ is continuous for all $q\in[1,p^*_s]$ and compact for all $q\in[1,p^*_s)$, with embedding constant $c_q>0$.
\vskip2pt
\noindent
We denote $\widetilde{W}^{s,p}(\Omega)$ the space of all $u \in L_{\rm loc}^p(\R^N)$ s.t.\ $u \in W^{s,p}(U)$ for some open $U \subseteq \R^N$, $\overline\Omega\subset U$, and
\[\int_{\R^N} \frac{|u(x)|^{p-1}}{(1+|x|)^{N+ps}}\,dx < \infty .\]
Clearly, $\w \subset \widetilde{W}^{s,p}(\Omega)$. By \cite[Lemma 2.3]{IMS1}, for any $u \in \widetilde{W}^{s,p}(\Omega)$ we can define $\fpl u \in W^{-s,p'}(\Omega)$ by setting for all $v\in\w$
\[\langle \fpl u, v\rangle = \iint_{\R^N\times\R^N} |u(x)-u(y)|^{p-1}(u(x)-u(y))(v(x)-v(y)) \,d\mu.\]
 The definition above agrees with \eqref{fpl} when $u$ lies in the Schwartz space of $C^\infty$, rapidly decaying functions in $\R^N$. In the next lemma we recall some useful properties of $\fpl$ in $\w$:

\begin{lemma} \label{S+}
$\fpl: \w \rightarrow W^{-s,p'}(\Omega)$ is a monotone, continuous, $(S)_+$-operator.
\end{lemma}
\begin{proof}
By \cite[Lemma 2.3]{IMS3} (with $q=1$) we have for all $u,v \in \w$
\[ \langle \fpl u-\fpl v, u-v\rangle \ge 0,\]
hence $\fpl$ is monotone. Plus, $\fpl$ is continuous as the G\^ateaux derivative of the $C^1$-functional $u \mapsto \frac{\|u\|_{s,p}^p}{p}$.
Finally, if $u_n \rightharpoonup u$ in $\w$ and
\[\limsup_n \langle \fpl u_n, u_n-u\rangle \le 0,\]
then for all $n \in \N$ we have
\begin{align*}
&(\|u_n\|_{s,p}^{p-1} - \|u_n\|_{s,p}^{p-1}) (\|u\|_{s,p} - \|u\|_{s,p}) = \|u_n\|_{s,p}^p - \|u_n\|_{s,p}^{p-1} \|u\|_{s,p} - \|u_n\|_{s,p} \|u\|_{s,p}^{p-1}+\|u\|_{s,p}^p\\
& \leq \langle \fpl u_n,u_n\rangle - \langle \fpl u_n,u\rangle - \langle \fpl u, u_n\rangle + \langle \fpl u, u\rangle \\
& = \langle \fpl u_n, u_n - u\rangle + \langle \fpl u, u - u_n\rangle \leq {\bf o}(1),
\end{align*}
hence $\|u_n\|_{s,p} \rightarrow \|u\|_{s,p}$. By uniform convexity of $\w$, $u_n \rightarrow u$ in $\w$. Therefore, $\fpl$ is an $(S)_+$-operator.
 \end{proof}

\noindent
Now we introduce basic hypothesis on the reaction $f$:
\begin{itemize}[leftmargin=1cm]
\item[${\bf H}_0$] $f:\Omega\times\R\to\R$ is a Carath\'{e}odory function s.t.\ for a.e.\ $x\in\Omega$ and all $t\in\R$ 
\[|f(x,t)| \leq c_0(1+|t|^{q-1}) \qquad (c_0 >0, q\in(p,p_s^*))\]
\end{itemize}

\noindent 
We recall some definitions:

\begin{definition}\label{ss}
Let $u \in \widetilde{W}^{s,p}(\Omega)$:
\begin{enumroman}
\item $u$ is a supersolution of \eqref{p} if $u \ge 0$ in $\Omega^c$ and for all $v \in \w_+$
\[\langle \fpl u,v\rangle \ge \int_{\Omega} f(x,u)v\,dx;\]
\item $u$ is a subsolution of \eqref{p} if $u \le 0$ in $\Omega^c$ and for all $v \in \w_+$
\[\langle \fpl u,v\rangle \le \int_{\Omega} f(x,u)v\,dx.\]
\end{enumroman}
\end{definition}

\noindent
We say that $(\underline{u}, \overline{u}) \in \widetilde{W}^{s,p}(\Omega) \times \widetilde{W}^{s,p}(\Omega)$ is a {\em sub-supersolution pair} of \eqref{p}, if $\underline{u}$ is a subsolution, $\overline{u}$ is a supersolution, and $\underline{u} \le \overline{u}$ in $\Omega$.

\begin{definition}\label{sol}
$u \in \w$ is a solution of \eqref{p} if for all $v \in \w$
\[\langle \fpl u,v\rangle =  \int_{\Omega} f(x,u)v\,dx .\]
 \end{definition}
 
\noindent
Clearly, $u\in\w$ is a solution of \eqref{p} iff it is both a supersolution and a subsolution. Sub-, supersolutions, and solutions of similar problems will be meant in the same sense as in Definitions \ref{ss}, \ref{sol} above. 
\vskip2pt
\noindent
We will need the following a priori bound for solutions of \eqref{p}:

\begin{lemma}\label{bound}
{\rm\cite[Lemma 2.1]{IMS3}} Let ${\bf H}_0$ hold, $u \in \w$ be a solution of \eqref{p}. Then, $u \in L^{\infty}(\Omega)$ with $\|u\|_{\infty} \le C$, for some $C=C(\|u\|_{s,p})>0$.
\end{lemma}

\noindent
We define weighted H\"{o}lder-type spaces with weight $\ds(x)= \mathrm{dist}(x, \Omega^c)^s$, along with their norms:
\[C_s^0(\overline{\Omega})= \Big\{u \in C^0(\overline{\Omega}): \frac{u}{\ds} \in C^0(\overline{\Omega}) \Big\},
\quad \|u\|_{0,s}= \Big\|\frac{u}{\ds}\Big\|_{\infty},\]
and for all $\alpha \in (0,1)$
\[C_s^{\alpha}(\overline{\Omega})= \Big\{u \in C^0(\overline{\Omega}): \frac{u}{\ds} \in C^{\alpha}(\overline{\Omega})\Big\},
\quad 
\|u\|_{\alpha,s}= \|u\|_{0,s} + \sup_{x \neq y} \frac{|u(x)/\ds(x) - u(y)/\ds(y)|}{|x-y|^{\alpha}}.\]
The embedding $C_s^{\alpha}(\overline{\Omega}) \hookrightarrow C_s^0(\overline{\Omega})$ is compact for all
$\alpha \in (0,1)$. Unlike in $\w$, the positive cone $C_s^0(\overline{\Omega})_+$ of $C_s^0(\overline{\Omega})$ has a nonempty interior given by
\begin{equation}\label{int}
\mathrm{int}(C_s^0(\overline{\Omega})_+)= \Big\{u \in C_s^0(\overline{\Omega}):\, \frac{u(x)}{\ds(x)} > 0 \text{ in } \overline{\Omega}\Big\}
\end{equation}
(see \cite[Lemma 5.1]{ILPS}).
Consider the following Dirichlet problem, with right-hand side $g\in L^\infty(\Omega)$:
\beq\label{p1}
\begin{cases}
\fpl u = g(x)& \text{in $\Omega$} \\
u=0 & \text{in $\Omega^c$.}
\end{cases}
\eeq
We have the following regularity result:

\begin{lemma}\label{reg}
{\rm\cite[Theorem 1.1]{IMS2}} Let $g \in L^{\infty}(\Omega)$, $u \in \w$ be a solution of \eqref{p1}. Then, $u \in C_s^{\alpha}(\overline{\Omega})$ with $\|u\|_{\alpha, s} \le C \|g\|_{\infty}^{\frac{1}{p-1}}$, for some $\alpha \in (0,s]$, $C=C(\Omega) >0$.
\end{lemma}

\noindent
Combining Lemmas \ref{bound}, \ref{reg} we see that any solution of $\eqref{p}$ under ${\bf H}_0$ lies in $C_s^{\alpha}(\overline{\Omega})$, with a uniform estimate on the $C_s^{\alpha}(\overline{\Omega})$-norm. In the final part of our study, we will follow a variational approach. We define an energy functional for problem \eqref{p} by setting for all $(x,t)\in\Omega\times\R$
\[F(x,t)=\int_0^t f(x,\tau)\,d\tau,\]
and for all $u\in \w$
\[\Phi(u)= \frac{\|u\|_{s,p}^p}{p} -  \int_{\Omega} F(x,u)\,dx.\]
By ${\bf H}_0$, it is easily seen that $\Phi \in C^1(\w)$ and the solutions of \eqref{p} coincide with the critical points of $\Phi$. We will need the following equivalence result for local minimizers of $\Phi$ in $\w$ and in $C_s^0(\overline\Omega)$:

\begin{lemma}\label{svh}
{\rm \cite[Theorem 1.1]{IMS3}} Let ${\bf H}_0$ hold, $u\in\w$. Then, the following are equivalent:
\begin{enumroman}
\item\label{svh1} there exists $\rho>0$ s.t.\ $\Phi(u+v)\ge\Phi(u)$ for all $v\in\w$, $\|v\|_{s,p}\le\rho$;
\item\label{svh2} there exists $\sigma>0$ s.t.\ $\Phi(u+v)\ge\Phi(u)$ for all $v\in\w\cap C_s^0(\overline\Omega)$, $\|v\|_{0,s}\le\sigma$.
\end{enumroman}
\end{lemma}

\noindent
Since we are mainly interested in constant sign solutions, we will need a strong maximum principle and Hopf's lemma. Consider the problem
\beq\label{p2}
\begin{cases}
\fpl u = -c(x) |u|^{p-2}u & \text{in $\Omega$} \\
u=0 & \text{in $\Omega^c$,}
\end{cases}
\eeq
with $c \in C^0(\overline{\Omega})_+$. Then we have the following:

\begin{lemma} \label{DQ}
{\rm\cite[Theorem 1.5]{DQ}} Let $c \in C^0(\overline{\Omega})_+$, $u \in \widetilde{W}^{s,p}(\Omega)_+ \setminus \{0\}$ be a supersolution of \eqref{p2}. Then, $u >0$ in $\Omega$ and for any 
$x_0 \in \partial \Omega$
\[\liminf_{\Omega \ni x \to x_0} \frac{u(x)}{\ds(x)} > 0.\]
\end{lemma}

\noindent
Finally, we recall some spectral properties of $\fpl$(see \cite{DQ1,HS} and \cite[Proposition 3.4]{FRS}).
Let $\rho \in L^{\infty}(\Omega)_+ \setminus \{0\}$ and consider the following weighted eigenvalue problem:
\begin{equation}
\begin{cases}
\fpl u = \lambda \rho(x) |u|^{p-2} u & \text {in $\Omega$}\\
u=0 & \text{on $\Omega^c$.}
\end{cases}
\label{p4}
\end{equation}

\begin{lemma}\label{sp}
Let $\rho \in L^{\infty}(\Omega)_{+} \setminus \{0\}$. Then, \eqref{p4} has an unbounded sequence of variational eigenvalues
\[0 < \lambda_1(\rho) < \lambda_2(\rho) \le \ldots \le \lambda_k(\rho) \le \ldots\]
The first eigenvalue admits the following variational characterization:
\[\lambda_1(\rho)= \inf_{u \in \w \setminus \{0\}} \frac{\|u\|_{s,p}^p}{\int_{\Omega} \rho(x)|u|^p\,dx},\]
and
\begin{enumroman}
\item\label{sp1} $\lambda_1(\rho)>0$ is simple, isolated and attained at an unique positive eigenfunction $\hat{u}_1(\rho) \in \w\cap\mathrm{int}(C_s^0(\overline{\Omega})_+)$ s.t.\ $\int_{\Omega} \rho(x)|\hat{u}_1|^p\,dx=1$;
\item\label{sp2} if $u \in \w \setminus \{0\}$ is an eigenfunction of \eqref{p4} associated to any eigenvalue $\lambda > \lambda_1(\rho)$, then $u$ is nodal;
\item\label{sp3} if $\tilde{\rho} \in L^{\infty}(\Omega)_+ \setminus \{0\}$ is s.t.\ $\tilde{\rho} \le \rho$, $\tilde{\rho} \not\equiv \rho$, then $\lambda_1(\rho) <  \lambda_1 (\tilde{\rho})$.
\end{enumroman}
\end{lemma}

\noindent
When $\rho \equiv 1$ we set $\lambda_1(\rho)=\lambda_1$ and $\hat{u}_1(\rho)=\hat{u}_1$. Moreover, the second (non-weighted) eigenvalue admits the following variational characterization:
\begin{equation}
\lambda_2 = \inf_{\gamma \in \Gamma_1} \max_{t \in [0,1]} \|\gamma(t)\|_{s,p}^p,
\label{e2}
\end{equation}
where
\[\Gamma_1 = \big\{\gamma \in C([0,1],\w):\,\gamma(0)=\hat{u}_1, \gamma(1)=- \hat{u}_1, \|\gamma(t)\|_p=1 \, \text{for all} \ t \in [0,1]\big\},\]
see \cite[Theorem 5.3]{BP}.

\section{Solutions in a sub-supersolution interval}\label{sec3}

\noindent In this section we consider a sub-supersolution pair $(\underline{u}, \overline{u})$ and study the set
\[\mathcal{S}(\underline{u}, \overline{u})=\{u \in \w: u \text{ solves } \eqref{p}, \underline{u} \le u \le \overline{u}\}.\]
On spaces $\w, \widetilde{W}^{s,p}(\Omega)$ we consider the partial pointwise order, inducing a lattice structure. We set $u \land v = \min\{u,v\}$ and $u \vee v =\max\{u,v\}$.
\vskip2pt
\noindent
The first result shows that the pointwise minimum of supersolutions is a supersolution, as well as the maximum of subsolutions is a subsolution (we give the proof in full detail, as it requires some careful calculations):
 
\begin{lemma}\label{Mm}
Let ${\bf H}_0$ hold and $u_1,u_2 \in \widetilde{W}^{s,p}(\Omega)$:
\begin{enumroman}
\item \label{S} if $u_1, u_2$ are supersolutions of \eqref{p}, then so is $u_1 \land u_2$;
\item \label{s} if $u_1, u_2$ are subsolutions of \eqref{p} then so is $u_1 \vee u_2$.
\end{enumroman}
\end{lemma}

\begin{proof}
We prove \ref{S}. We have for $i=1,2$
\begin{equation}
\begin{cases}
\langle \fpl u_i , v\rangle \ge \int_{\Omega} f(x,u_i)v\,dx & \text {for all $v \in \w_+$}\\
u_i \ge 0 & \text{in $\Omega^c$.}
\end{cases}
\label{p5}
\end{equation}
Set $u=u_1 \land u_2 \in \widetilde{W}^{s,p}(\Omega)$ (by the lattice structure of $\widetilde{W}^{s,p}(\Omega)$), then $u\ge 0$ in $\Omega^c$. Set also
\[A_1=\{x \in \R^N:\,u_1(x) < u_2(x)\}, \quad A_2= A_1^c.\]
Now fix $\varphi \in C_c^{\infty}(\Omega)_+$, $\eps >0$, and set for all $t\in \R$
\[\tau_{\eps}(t)=
\begin{cases}
0 \ & \text{ if } t \le 0 \\
\displaystyle\frac{t}{\eps} \ & \text{ if } 0< t < \eps\\
1 & \text{ if } t \ge \eps.
\end{cases}
\]
The mapping $\tau_{\eps}: \R \to \R$ is Lipschitz continuous, nondecreasing, and $0\le\tau_\eps(t)\le 1$ for all $t\in\R$, and clearly
 \[\tau_{\eps}(u_2-u_1) \rightarrow \chi_{A_1}, \qquad  1-\tau_{\eps}(u_2-u_1) \rightarrow \chi_{A_2}\]
a.e.\ in $\R^N$, as $\eps \to 0^+$, with dominated convergence. Testing \eqref{p5} with $\tau_{\eps}(u_2-u_1)\varphi, (1-\tau_{\eps}(u_2-u_1))\varphi \in \w_+$ for $i=1,2$ respectively, we get
\begin{align}
&\langle \fpl u_1,\tau_{\eps}(u_2-u_1)\varphi \rangle + \langle \fpl u_2,(1-\tau_{\eps}(u_2-u_1))\varphi \rangle \label{Ss} \\
&\ge \int_{\Omega} f(x,u_1)\tau_{\eps}(u_2-u_1)\varphi \,dx + \int_{\Omega} f(x,u_2)(1-\tau_{\eps}(u_2-u_1))\varphi \,dx. \nonumber
\end{align}
We focus on the left-hand side of \eqref{Ss}. Setting for brevity $\tau_{\eps}=\tau_{\eps}(u_2-u_1)$ and $a^{p-1}=|a|^{p-2}a$ for all $a\in\R$, and recalling that $\tau_{\eps}=0$ in $A_2$, while $\tau_{\eps} \rightarrow 1$ in $A_1$ as $\eps \to 0^+$, we get
\begin{align*}
&\langle \fpl u_1,\tau_{\eps} \varphi\rangle + \langle \fpl u_2,(1-\tau_{\eps})\varphi\rangle \label{sum} \\
&= \iint_{\R^N\times\R^N}(u_1(x)-u_1(y))^{p-1}(\tau_{\eps}(x)\varphi(x)-\tau_{\eps}(y)\varphi(y))\,d\mu \\
&+ \iint_{\R^N\times\R^N}(u_2(x)-u_2(y))^{p-1}[(1-\tau_{\eps}(x))\varphi(x)-(1-\tau_{\eps}(y))\varphi(y)]\,d\mu \\
&=: I.
\end{align*}
Using the definition of $A_1$ and $A_2$, we obtain
\reqnomode
\begin{align*}
I &= \iint_{A_1\times A_1}(u_1(x)-u_1(y))^{p-1}(\varphi(x)-\varphi(y))\tau_{\eps}(x)\,d\mu \tag{A} \\
&+  \iint_{A_1\times A_1}(u_1(x)-u_1(y))^{p-1}\varphi(y)(\tau_{\eps}(x)-\tau_{\eps}(y))\,d\mu \tag{B} \\
&+\iint_{A_1\times A_2}(u_1(x)-u_1(y))^{p-1}\varphi(x)\tau_{\eps}(x)\,d\mu \tag{C} \\
&- \iint_{A_2 \times A_1}(u_1(x)-u_1(y))^{p-1}\varphi(y)\tau_{\eps}(y)\,d\mu \tag{D}\\
&+\iint_{A_1\times A_1}(u_2(x)-u_2(y))^{p-1}(\varphi(x)-\varphi(y))(1-\tau_{\eps}(x))\,d\mu \tag{E}\\
&- \iint_{A_1\times A_1}(u_2(x)-u_2(y))^{p-1}\varphi(y)(\tau_{\eps}(x)-\tau_{\eps}(y))\,d\mu \tag{B}\\
&+\iint_{A_1\times A_2}(u_2(x)-u_2(y))^{p-1}(\varphi(x)-\varphi(y))(1-\tau_{\eps}(x))\,d\mu \tag{F}\\
&- \iint_{A_1\times A_2}(u_2(x)-u_2(y))^{p-1}\varphi(y)\tau_{\eps}(x)\,d\mu \tag{C}\\
&+ \iint_{A_2\times A_1}(u_2(x)-u_2(y))^{p-1}\varphi(x)\tau_{\eps}(y)\,d\mu \tag{D}\\
&+ \iint_{A_2\times A_1}(u_2(x)-u_2(y))^{p-1}(\varphi(x)-\varphi(y))(1-\tau_{\eps}(y))\,d\mu \tag{G}\\
&+ \iint_{A_2\times A_2}(u_2(x)-u_2(y))^{p-1}(\varphi(x)-\varphi(y))\,d\mu \tag{H}.
\end{align*}
We then put together the integrals with the same letter and note that $({\rm E}),({\rm F}),({\rm G})\to 0$ as $\eps \to 0^+$. So, we have
\begin{align*}
I &= \iint_{A_1\times A_1}(u_1(x)-u_1(y))^{p-1}(\varphi(x)-\varphi(y))\,d\mu \tag{A}\\
&+  \iint_{A_1\times A_1}[(u_1(x)-u_1(y))^{p-1}-(u_2(x)-u_2(y))^{p-1}] \varphi(y)(\tau_{\eps}(x)-\tau_{\eps}(y))\,d\mu \tag{B}\\
&+ \iint_{A_1\times A_2}[(u_1(x)-u_1(y))^{p-1}\varphi(x)-(u_2(x)-u_2(y))^{p-1}\varphi(y)] \tau_{\eps}(x)\,d\mu \tag{C}\\
&+ \iint_{A_2\times A_1}[(u_2(x)-u_2(y))^{p-1}\varphi(x)-(u_1(x)-u_1(y))^{p-1}\varphi(y)] \tau_{\eps}(y)\,d\mu \tag{D}\\
&+ \iint_{A_2\times A_2}(u_2(x)-u_2(y))^{p-1}(\varphi(x)-\varphi(y))\,d\mu \tag{H}\\
&+ {\bf o}(1).
\end{align*}
Now we note that for all $x,y\in A_1$
\[u_1(x)-u_1(y) \ge u_2(x)-u_2(y) \ \Leftrightarrow \ u_2(y)-u_1(y) \ge u_2(x) - u_1(x) \ \Leftrightarrow \ \tau_{\eps}(y) \ge \tau_{\eps}(x),\]
hence the integrand in $({\rm B})$ is negative. Besides, for all $x\in A_1$, $y\in A_2$
\[u_1(x)-u_1(y) \le u_1(x)-u_2(y) \le u_2(x)-u_2(y),\]
and for all $x\in A_2$, $y\in A_1$
\[u_2(x)-u_2(y) \le u_2(x)-u_1(y) \le u_1(x)-u_1(y),\]
so we can estimate the integrands in $({\rm C})$, $({\rm D})$ respectively and get
\begin{align*}
I &\le \iint_{A_1\times A_1}(u_1(x)-u_1(y))^{p-1}(\varphi(x)-\varphi(y))\,d\mu \\
&+ \iint_{A_1\times A_2}(u_1(x)-u_2(y))^{p-1}(\varphi(x)-\varphi(y))\,d\mu\\
&+ \iint_{A_2\times A_1}[(u_2(x)-u_1(y))^{p-1}(\varphi(x)-\varphi(y))\,d\mu \\
&+ \iint_{A_2\times A_2}(u_2(x)-u_2(y))^{p-1}(\varphi(x)-\varphi(y))\,d\mu + {\bf o}(1)\\
&= \langle \fpl u,\varphi \rangle + {\bf o}(1).
\end{align*}
\leqnomode
All in all, we have
\beq\label{lhs}
\langle \fpl u_1,\tau_{\eps}(u_2-u_1)\varphi \rangle + \langle \fpl u_2,(1-\tau_{\eps}(u_2-u_1))\varphi \rangle \le \langle \fpl u,\varphi \rangle + {\bf o}(1),
\eeq
as $\eps\to 0^+$. Regarding the right-hand side of \eqref{Ss}, we use the bounds from ${\bf H}_0$ and the definition of $\tau_\eps$ to get
\[|f(\cdot, u_1)\tau_{\eps}^+(u_2-u_1)\varphi| \le c_0 ( 1 + |u_1|^{q-1}) \varphi,\]
\[|f(\cdot, u_2)(1-\tau_{\eps}^+(u_2-u_1))\varphi| \le c_0 ( 1 + |u_2|^{q-1}) \varphi,\]
and pass to the limit as $\eps \to 0^+$:
\begin{align}
&\int_{\Omega} f(x,u_1)\tau_{\eps}(u_2-u_1)\varphi\,dx + \int_{\Omega} f(x,u_2)(1-\tau_{\eps}(u_2-u_1))\varphi\,dx \label{f}\\
&= \int_{\Omega} f(x,u_1)\chi_{A_1}\varphi\,dx + \int_{\Omega} f(x,u_2)\chi_{A_2}\varphi\,dx +{\bf o}(1) \nonumber \\
&= \int_{\Omega} f(x,u)\varphi\,dx+{\bf o}(1). \nonumber
\end{align}
Plugging \eqref{lhs}, \eqref{f} into \eqref{Ss} we have for all $\varphi \in C_c^{\infty}(\overline{\Omega})_+$
\[\langle \fpl u, \varphi \rangle \ge \int_{\Omega} f(x,u)\varphi\,dx.\]
By density, the same holds with test functions in $W_0^{s,p}(\Omega)_+$, hence $u$ is a supersolution of \eqref{p}, which proves \ref{S}. Similarly we prove \ref{s}.
\end{proof}

\noindent
Now we consider a sub-supersolution pair $(\underline{u}, \overline{u})$ and we study the set $\mathcal{S}(\underline{u}, \overline{u})$. We begin with a sub-supersolution principle, showing that $\mathcal{S}(\underline{u},\overline{u})\neq\emptyset$:

\begin{lemma}\label{Ex}
Let ${\bf H}_0$ hold and $(\underline{u}, \overline{u})$ be a sub-supersolution pair of \eqref{p}. Then, there exists $u \in \mathcal{S}(\underline{u}, \overline{u})$.
\end{lemma}
\begin{proof}
In this argument we use some nonlinear operator theory from \cite{CLM}. First we define $A=\fpl:\w\to W^{-s,p'}(\Omega)$. By Lemma \ref{S+} $A$ is monotone and continuous, hence hemicontinuous \cite[Definition 2.95 $(iii)$]{CLM}, therefore $A$ is pseudomonotone \cite[Lemma 2.98 $(i)$]{CLM}.
\vskip2pt
\noindent
Besides, we set for all $(x,t)\in\Omega\times\R$
\[\tilde{f}(x,t)=
\begin{cases}
f(x,\underline{u}(x)) \ & \text{ if } t \le \underline{u}(x) \\
f(x,t) \ & \text{ if } \underline{u}(x) < t < \overline{u}(x) \\
f(x, \overline{u}(x)) & \text{ if } t \ge \overline{u}(x).
\end{cases}\]
In general, $\tilde{f}$ does not satisfy ${\bf H}_0$, but still $\tilde{f}:\Omega\times\R\to\R$ is a Carath\'{e}odory function s.t.\ for a.e.\ $x\in\Omega$ and all $t\in\R$
\begin{equation}
|\tilde{f}(x,t)| \le c_0 ( 1 + |\underline{u}|^{q-1} + |\overline{u}|^{q-1}) .
\label{ft}
\end{equation}
We define $B: \w \rightarrow W^{-s, p'}(\Omega)$ by setting for all $u,v \in \w$
\[\langle B(u),v\rangle = - \int_{\Omega} \tilde{f}(x,u)v\,dx,\]
well posed by \eqref{ft}, as $|\underline{u}|^{q-1}, |\overline{u}|^{q-1} \in L^{q'}(\Omega)$. We prove that $B$ is strongly continuous \cite[Definition 2.95 $(iv)$]{CLM}. Indeed, let $(u_n)$ be a sequence s.t.\ $u_n \rightharpoonup u$ in $\w$, passing to a subsequence if necessary, we have $u_n \rightarrow u$ in $L^q(\Omega)$, $u_n(x) \rightarrow u(x)$ and $|u_n(x)| \le h(x)$ for a.e.\ $x \in \Omega$, for some $h \in L^q(\Omega)$. Therefore, for all $n \in \N$, by \eqref{ft} we have for a.e.\ $x\in\Omega$
\[|\tilde{f}(x,u_n) - \tilde{f}(x,u)| \le 2 c_0 ( 1 + |\underline{u}|^{q-1} + |\overline{u}|^{q-1}) \in L^{q'}(\Omega),\]
while by continuity of $f(x,\cdot)$ we have $\tilde{f}(x,u_n) \to \tilde{f}(x,u)$. Hence, for all $v\in \w$,
\begin{align*}
|\langle B(u_n) - B(u),v\rangle| &\le \int_{\Omega} |\tilde{f}(x,u_n)- \tilde{f}(x,u)| |v|\,dx \\
&\le \|\tilde{f}(\cdot,u_n)- \tilde{f}(\cdot,u)\|_{q'} \|v\|_q
\end{align*}
and the latter tends to $0$ as $n \to \infty$, uniformly with respect to $v$. Therefore $B(u_n) \rightarrow B(u)$ in $W^{-s,p'}(\Omega)$. By \cite[Lemma 2.98  $(ii)$]{CLM}, $B$ is pseudomonotone. Thus, $A+B$ is pseudomonotone.
\vskip2pt
\noindent
Now we prove that $A+B$ is bounded. Indeed, for all $u \in \w$ we have $\|A(u)\|_{-s,p'} \le \|u\|_{s,p}^{p-1}$ and
\begin{align*}
\|B(u)\|_{-s,p'} &= \sup_{\|v\|_{s,p} \le 1} \int_{\Omega} \tilde{f}(x,u) v\,dx \\
&\le C\|\tilde{f}(\cdot,u)\|_{q'} \\
&\le C (1 + \|\underline{u}\|_q^{q-1} + \|\overline{u}\|_q^{q-1}),
\end{align*}
where we have used \eqref{ft} and the continuous embedding $\w\hookrightarrow L^q(\Omega)$.
\vskip2pt
\noindent
Finally we prove that $A+B$ is coercive. Indeed, for all $u \in \w \setminus \{0\}$ we have
\begin{align*}
\frac{\langle A(u) + B(u),u\rangle} {\|u\|_{s,p}} &=  \|u\|_{s,p}^{p-1} - \frac{1}{\|u\|_{s,p}} \int_{\Omega} \tilde{f}(x,u) u\,dx \\
& \ge \|u\|_{s,p}^{p-1} - \frac{C}{\|u\|_{s,p}} \int_{\Omega} (1 + |\underline{u}|^{q-1} + |\overline{u}|^{q-1}) |u|\,dx \\
& \ge \|u\|_{s,p}^{p-1} - \frac{C}{\|u\|_{s,p}} \left(\|u\|_1 + \|\underline{u}\|_q^{q-1} \|u\|_q+ \|\overline{u}\|_q^{q-1} \|u\|_q \right)\\
& \ge \|u\|_{s,p}^{p-1} - C,
\end{align*}
and the latter tends to $\infty$ as $\|u\|_{s,p} \to \infty$ (here we have used the continuous embeddings $\w\hookrightarrow L^1(\Omega),\,L^q(\Omega)$). By \cite[Theorem 2.99]{CLM}, the equation
\begin{equation}
A(u)+B(u)=0 \text{ in } W^{-s,p'}(\Omega)
\label{eq}
\end{equation}
has a solution $u\in\w$. Now we prove that in $\Omega$ 
\begin{equation}
\underline{u} \le u \le \overline{u}.
\label{int1}
\end{equation}
Clearly \eqref{int1} holds in $\Omega^c$. Testing \eqref{eq} with $(u-\overline{u})^+ \in \w_+$ we have
\begin{align*}
\langle \fpl u, (u-\overline{u})^+\rangle &= \int_{\Omega} \tilde{f}(x,u) (u-\overline{u})^+ \,dx \\
&= \int_{\Omega} f(x, \overline{u}) (u-\overline{u})^+ \,dx \\
&\le \langle \fpl \overline{u}, (u-\overline{u})^+\rangle,
\end{align*}
where we also used that $\overline{u}$ is a supersolution of \eqref{p}, so
\[\langle \fpl u - \fpl\overline{u}, (u-\overline{u})^+\rangle \le 0.\]
By \cite[Lemma A.2]{BP} and \cite[Lemma 2.3]{IMS3} (with $g(t)=t^+$) we have for all $a,b \in \R$
\[|a^+-b^+|^p \le (a-b)^{p-1} (a^+ - b^+), \quad (a-b)^{p-1} \le C (a^{p-1}-b^{p-1}),\]
hence
\begin{align*}
&\|(u-\overline{u})^+\|_{s,p}^p = \iint_{\R^N\times\R^N} |(u(x)-\overline{u}(x))^+ - (u(y)-\overline{u}(y))^+|^p\,d\mu\\
& \le \iint_{\R^N\times\R^N} [(u(x)-\overline{u}(x)) - (u(y)-\overline{u}(y))]^{p-1}[(u(x)-\overline{u}(x))^+ - (u(y)-\overline{u}(y))^+]\,d\mu\\
& \le C \iint_{\R^N\times\R^N} [(u(x) - u(y))^{p-1}-(\overline{u}(x)-\overline{u}(y))^{p-1}]  [(u(x)-\overline{u}(x))^+ - (u(y)-\overline{u}(y))^+]\,d\mu\\
&=C \langle \fpl u - \fpl \overline{u}, (u-\overline{u})^+\rangle \le 0,
\end{align*}
so $(u-\overline{u})^+ = 0$, i.e., $u \le \overline{u}$ in $\Omega$. Similarly we prove $u \ge \underline{u}$ and achieve \eqref{int1}. Finally, using \eqref{int1} in \eqref{eq} we see that $u \in \w$ solves \eqref{p}. Thus $u\in \mathcal{S}(\underline{u}, \overline{u})$.
\end{proof}

\noindent
We recall that a partially ordered set $(S,\le)$ is {\em downward directed} (resp., {\em upward directed}) if for all $u_1,u_2\in S$ there exists $u_3\in S$ s.t.\ $u_3\le u_1,u_2$ (resp., $u_3\ge u_1,u_2$), and that $S$ is {\em directed} if it is both downward and upward directed.

\begin{lemma}\label{dir}
Let ${\bf H}_0$ hold, $(\underline{u}, \overline{u})$ be a sub-supersolution pair of \eqref{p}. Then, $\mathcal{S}(\underline{u}, \overline{u})$ is directed.
\end{lemma}
\begin{proof}
We prove that $\mathcal{S}(\underline{u}, \overline{u})$ is downward directed. Let $u_1, u_2 \in \mathcal{S}(\underline{u}, \overline{u})$, then in particular $u_1, u_2$ are supersolutions of \eqref{p}. Set $\hat{u}=u_1 \land u_2 \in \w$, then by Lemma \ref{Mm} $\hat{u}$ is a supersolution of \eqref{p} and $\underline{u} \le \hat{u}$.
By Lemma \ref{Ex} there exists $u_3 \in \mathcal{S}(\underline{u}, \hat{u})$, in particular $u_3 \in \mathcal{S}(\underline{u}, \overline{u})$ and $u_3 \le u_1 \land u_2$.
\vskip2pt
\noindent
Similarly we see that $\mathcal{S}(\underline{u}, \overline{u})$ is upward directed.
\end{proof}

\noindent
Another important property of $\mathcal{S}(\underline{u},\overline{u})$ is compactness:

\begin{lemma}\label{comp}
Let ${\bf H}_0$ hold, $(\underline{u}, \overline{u})$ be a sub-supersolution pair of \eqref{p}. Then, $\mathcal{S}(\underline{u}, \overline{u})$ is compact in $\w$.
\end{lemma}
\begin{proof}
Let $(u_n)$ be a sequence in $\mathcal{S}(\underline{u}, \overline{u})$, then for all $n\in \N$, $v \in \w$
\begin{equation}
\langle \fpl u_n,v\rangle= \int_{\Omega} f(x,u_n)v\,dx
\label{c1}
\end{equation}
and $\underline{u} \le u_n \le \overline{u}$. Testing \eqref{c1} with $u_n \in \w$, we have by ${\bf H}_0$
\begin{align*}
\|u_n\|_{s,p}^p &= \int_{\Omega} f(x,u_n)u_n\,dx \\
&\le c_0 \int_{\Omega} (|u_n|+|u_n|^q)\,dx \\
&\le c_0(\|\underline{u}\|_1 + \|\overline{u}\|_1 + \|\underline{u}\|_q^q + \|\overline{u}\|_q^q)\le C,
\end{align*}
hence $(u_n)$ is bounded in $\w$. Passing to a subsequence, we have $u_n \rightharpoonup u$ in $\w$, $u_n(x) \rightarrow u(x)$ and $|u_n(x)|\le h(x)$ for a.e.\ $x\in\N$, with $h\in L^q(\Omega)$. Therefore,
\begin{align*}
|f(x,u_n)(u_n-u)| &\le c_0 (1+|u_n|^{q-1}) |u_n-u| \\
&\le 2c_0 (1+g(x)^{q-1}) (|\underline{u}| + |\overline{u}|) \in L^1(\Omega).
\end{align*}
Testing \eqref{c1} with $u_n - u \in \w$, we get 
\[\langle \fpl(u_n),u_n-u\rangle= \int_{\Omega} f(x,u_n)(u_n-u)\,dx,\]
and the latter tends to $0$ as $n\to\infty$. By Lemma \ref{S+} we have $u_n \rightarrow u$ in $\w$. Then, we can pass to the limit in \eqref{c1} and conclude that $u \in \mathcal{S}(\underline{u}, \overline{u})$.
\end{proof}

\noindent
The main result of this section states that $\mathcal{S}(\underline{u}, \overline{u})$ contains extremal elements with respect to the pointwise ordering:

\begin{theorem}\label{SG}
Let ${\bf H}_0$ hold, $(\underline{u}, \overline{u})$ be a sub-supersolution pair of \eqref{p}. Then $\mathcal{S}(\underline{u}, \overline{u})$ contains a smallest and a biggest element.
\end{theorem}
\begin{proof}
The set $\mathcal{S}(\underline{u}, \overline{u})$ is bounded in both $\w$ and $C_s^{\alpha}(\overline{\Omega})$. Indeed, for all $u \in \mathcal{S}(\underline{u}, \overline{u})$, testing \eqref{p} with $u \in \w$ we have
\begin{align*}
\|u\|_{s,p}^p &= \int_{\Omega} f(x,u)u\,dx \\
&\le c_0 \int_{\Omega} (|u|+|u|^q)\,dx \\
&\le c_0(\|\underline{u}\|_1 + \|\overline{u}\|_1 + \|\underline{u}\|_q^q + \|\overline{u}\|_q^q),
\end{align*}
hence $\mathcal{S}(\underline{u}, \overline{u})$ is bounded in $\w$. Further, by Lemma \ref{bound}, for all $u\in\mathcal{S}(\underline{u},\overline{u})$ we have $u \in L^{\infty}(\Omega)$, $\|u\|_{\infty} \le C$ (with $C=C(\underline{u}, \overline{u}) >0$, here and in the forthcoming bounds). In turn, this implies $\|f(\cdot, u)\|_{\infty} \le C$. Then we apply Lemma \ref{reg} (with $g=f(\cdot,u)$) to see that $u \in C_s^{\alpha}(\overline{\Omega})$, $\|u\|_{\alpha,s}\le C$. So, $\mathcal{S}(\underline{u}, \overline{u})$ is bounded in $C_s^{\alpha}(\overline{\Omega})$ as well (in particular, then, $\mathcal{S}(\underline{u}, \overline{u})$ is equibounded in $\Omega$).
\vskip2pt
\noindent
Now we prove that $\mathcal{S}(\underline{u}, \overline{u})$ has a minimum. Let $(x_k)$ be a dense subset of $\Omega$, and set
\[m_k= \inf_{u \in \mathcal{S}(\underline{u}, \overline{u})} u(x_k) > - \infty \]
for each $k \ge 1$ (recall $\mathcal{S}(\underline{u}, \overline{u})$ is equibounded). For all $n\in \N$, $k \in \{1, \dots, n\}$ we can find $u_{n,k} \in \mathcal{S}(\underline{u}, \overline{u})$ s.t.\ 
\[u_{n,k}(x_k) \le  m_k + \frac{1}{n}.\]
Since $\mathcal{S}(\underline{u}, \overline{u})$ is downward directed (Lemma \ref{dir}), we can find  $u_n \in \mathcal{S}(\underline{u}, \overline{u})$ s.t.\ $u_n \le u_{n,k}$ for all $k \in \{1, \dots, n\}$. In particular, for all $n \in \N$, $k \in \{1, \dots, n\}$ we have
\begin{equation}
u_n(x_k) \le m_k + \frac{1}{n}.
\label{F}
\end{equation}
Since $\mathcal{S}(\underline{u}, \overline{u})$ is compact (Lemma \ref{comp}), passing to a subsequence we have $u_n \to u_0$ in $\w$ for some $u_0 \in \mathcal{S}(\underline{u}, \overline{u})$. Besides, $(u_n) \subseteq \mathcal{S}(\underline{u}, \overline{u})$ is bounded in $C_s^{\alpha}(\overline{\Omega})$, hence up to a further subsequence $u_n \to u_0$ in $C_s^0(\overline{\Omega})$, in particular $u_n(x) \to u_0(x)$ for all $x \in \overline{\Omega}$. By \eqref{F} we have for all $k \in \N$
\[u_0(x_k)= \lim_n u_n(x_k) \le \lim_n \left(m_k + \frac{1}{n}\right) = m_k.\]
Therefore, given $u \in \mathcal{S}(\underline{u}, \overline{u})$ we have $u_0(x_k) \le u(x_k)$ for all $k \ge 1$, which by density of $(x_k)$ implies $u_0 \le u$. Hence, 
\[u_0 = \min \mathcal{S}(\underline{u}, \overline{u}).\]
Similarly we prove the existence of $\max \mathcal{S}(\underline{u}, \overline{u})$.
\end{proof}

\begin{remark}
For the sake of completeness, we recall that Theorem \ref{SG} can be proved following closely the proof of \cite[Theorem 3.11]{CLM}, using Lemmas \ref{dir}, \ref{comp}, and the fact that $\w$ is separable (another way consists in applying Zorn's Lemma, as in \cite[Remark 3.12]{CLM}). We also note the remark that, as seen in the proof of Theorem \ref{SG}, $\mathcal{S}(\underline{u},\overline{u})$ turns out to be compact in $C^0_s(\overline\Omega)$.
\end{remark}

\section{Extremal constant sign solutions}\label{sec4}

\noindent
In this section we prove that \eqref{p} has a smallest positive and a biggest negative solution (following the ideas of \cite{CP}), under the following hypotheses on $f$:

\begin{itemize}[leftmargin=1cm]
\item[${\bf H}_1$] $f:\Omega\times\R\to\R$ is a Carath\'{e}odory function, for all $(x,t)\in\Omega\times\R$ we set
\[F(x,t)=\int_0^t  f(x,\tau)\,d\tau,\]
and the following conditions hold:
\begin{enumroman}
\item\label{h1} $|f(x,t)| \leq c_0(1+|t|^{q-1})$ for all a.e.\ $x\in\Omega$ and all $t\in\R$ ($c_0>0$, $q \in(p,p_s^*)$);
\item\label{h2} $\displaystyle\limsup_{|t|\to\infty} \frac{F(x,t)}{|t|^p}<\frac{\lambda_1}{p}$ uniformly for a.e.\ $x\in\Omega$;
\item\label{h3} $\displaystyle\lambda_1 < \liminf_{t\to 0}\frac{f(x,t)}{|t|^{p-2}t} \le \limsup_{t\to 0}\frac{f(x,t)}{|t|^{p-2}t} < \infty$ uniformly for a.e.\ $x\in\Omega$.
\end{enumroman}
\end{itemize}

\noindent
Clearly ${\bf H}_1$ implies ${\bf H}_0$. Here $\lambda_1 >0$ denotes the principal eigenvalue of $\fpl$ in $\w$, with associated positive, $L^p(\Omega)$-normalized eigenfunction $\hat{u}_1 \in \w$ (see Lemma \ref{sp} \ref{sp1}). Note that by ${\bf H}_1$ \ref{h3} we have $f(\cdot,0)=0$ in $\Omega$, hence \eqref{p} has the trivial solution $0$. Condition ${\bf H}_1$ \ref{h3} conjures a $(p-1)$-linear behavior of $f(x,\cdot)$ near the origin.
\vskip2pt
\noindent
In this and the forthcoming section, our approach to problem \eqref{p} is purely variational. Our result is the following:

\begin{theorem}\label{main}
Let ${\bf H}_1$ hold. Then, \eqref{p} has a smallest positive solution $u_{+} \in \mathrm{int}(C_s^0(\overline{\Omega})_+)$ 
and a biggest negative solution $u_{-} \in -\mathrm{int}(C_s^0(\overline{\Omega})_+)$.
\end{theorem}
\begin{proof}
We focus on positive solutions. Set for all $(x,t) \in \Omega \times \R$
\[f_+(x,t)=f(x,t^+), \quad   F_+(x,t)= \int_0^t  f_+(x,\tau)\,d\tau,\]
and for all $u \in \w$
\[\Phi_{+}(u)= \frac{\|u\|_{s,p}^p}{p} - \int_{\Omega} F_+(x,u)\,dx.\]
Since $f_+(x,t)=0$ for all $(x,t) \in \Omega \times \R^{-}$, $f_+$ satisfies ${\bf H}_1$ (with $t \to 0^+$ in \ref{h3}). Therefore, $\Phi_{+} \in C^1(\w)$. By ${\bf H}_1$ \ref{h1} and the compact embedding $\w \hookrightarrow L^q(\Omega)$, it is easily seen that $\Phi_{+}$ is sequentially weakly lower semicontinuous in $\w$.
\vskip2pt
\noindent
By ${\bf H}_1$ \ref{h2} there exist $\theta\in(0,\lambda_1)$, $K>0$ s.t.\ for a.e.\ $x\in\Omega$ and all $|t|\ge K$
\[F_+(x,t) \le \frac{\theta}{p}|t|^p.\]
Besides, by ${\bf H_1}$ \ref{h1} we can find $C_K>0$ s.t.\ for a.e.\ $x\in\Omega$ and all $t\in\R$
\[F_+(x,t) \le \frac{\theta}{p}|t|^p+C_K.\]
So, for all $u\in\w$ we have
\begin{align*}
\Phi_+(u) &\ge \frac{\|u\|_{s,p}^p}{p}-\int_\Omega\Big(\frac{\theta}{p}|u|^p+C_K\Big)\,dx \\
&\ge \frac{\|u\|_{s,p}^p}{p}-\frac{\theta}{p}\|u\|_p^p-C_K|\Omega| \\
&\ge \Big(1-\frac{\theta}{\lambda_1}\Big)\frac{\|u\|_{s,p}^p}{p}-C_K|\Omega|
\end{align*}
(where we used Lemma \ref{sp}), and the latter tends to infinity as $\|u\|_{s,p}\to \infty$. Therefore $\Phi_{+}$ is coercive. Thus, there is $\hat{u} \in \w$ s.t.\ 
\begin{equation}
\Phi_{+}(\hat{u})=\inf_{u \in \w} \Phi_{+}(u).
\label{min}
\end{equation}
In particular, we have $\Phi'_{+}(\hat{u})=0$, i.e.,
\begin{equation}
\fpl\hat{u}=f_+(\cdot, \hat{u}) \text{ in } W^{-s,p'}(\Omega).
\label{fd}
\end{equation}
Testing \eqref{fd} with $-\hat{u}^{-} \in \w$, we get 
\[\|\hat{u}^{-}\|^p \le - \langle \fpl\hat{u},\hat{u}^-\rangle= -\int_{\Omega} f_+(x,\hat{u})\hat{u}^-\,dx=0,\]
so $\hat{u} \ge 0$. Hence, $f_{+}(\cdot,\hat{u})=f(\cdot,\hat{u})$, therefore \eqref{fd} rephrases as
\[\fpl(\hat{u})=f(\cdot, \hat{u}) \text{ in } W^{-s,p'}(\Omega),\]
i.e., $\hat{u}\in\w_+$ is a solution of \eqref{p}. By Lemmas \ref{bound}, \ref{reg} we have $\hat{u} \in C_s^0(\overline{\Omega})_{+}$. By ${\bf H}_1$ \ref{h3}, we can find $\lambda_1 < c_1 < c_2$, $\delta>0$ s.t.\ for a.e.\ $x\in\Omega$ and all $t\in [0,\delta]$
\begin{equation}
c_1 t^{p-1} \le f(x,t) \le c_2 t^{p-1} .
\label{st}
\end{equation}
Choose $\tau>0$ s.t.\ $0<\tau \hat{u}_1 \le \delta$ in $\Omega$. Then by \eqref{min}, \eqref{st}, and Lemma \ref{sp} we have
\begin{align*}
\Phi_{+}(\hat{u}) &\le \Phi_{+}(\tau \hat{u}_1) \\
&=  \frac{\tau^p}{p} \|\hat{u}_1\|_{s,p}^p - \int_{\Omega} F_{+}(x, \tau \hat{u}_1)\,dx \\
&\le \frac{\tau^p}{p} \|\hat{u}_1\|_{s,p}^p - \frac{\tau^p c_1}{p} \|\hat{u}_1\|_p^p \\
&= \frac{\tau^p}{p} (\lambda_1 -c_1) <0,
\end{align*}
hence $\hat{u} \neq 0$. By \eqref{fd}, \eqref{st} we have for all $v \in \w_+$
\begin{align*}
\langle \fpl\hat{u}, v\rangle &\ge \int_{\{\hat{u} \le \delta\}} c_1 \hat{u}^{p-1} v\,dx - \int_{\{\hat{u} > \delta\}} c_0 (1+\hat{u}^{q-1}) v\,dx \\
&\ge \int_{\Omega} c_1 \hat{u}^{p-1} v\,dx - c_0 \int_{\{\hat{u} > \delta\}} \left[\frac{1}{\delta^{p-1}} + \|\hat{u}\|_{\infty}^{q-p} \right] \hat{u}^{p-1}v\,dx\\
&\ge -C  \int_{\Omega} \hat{u}^{p-1} v\,dx
\end{align*}
for some $C>0$. By Lemma \ref{DQ} and \eqref{int} we have $\hat{u} \in \mathrm{int}(C_s^0(\overline{\Omega})_+)$, so there is $r>0$ s.t.\ $u \in C_s^0(\overline{\Omega})_+$ for all $u \in C_s^0(\overline{\Omega})$ with $\|u-\hat{u}\|_{0,s} < r$. Now pick
\begin{equation}
0 < \eps < \min \Big\{\frac{\delta}{\|\hat{u}\|_{\infty}}, \frac{r}{\|\hat{u}_1\|_{0,s}}\Big\}.
\label{ep}
\end{equation} 
By \eqref{st} we have for all $v\in \w_+$
\[\langle \fpl(\eps \hat{u}_1), v\rangle = \lambda_1 \int_{\Omega}  (\eps \hat{u}_1)^{p-1} v\,dx \le \int_{\Omega}  f(x,\eps \hat{u}_1) v\,dx,\]
hence $\eps \hat{u}_1$ is a subsolution of \eqref{p}. Besides,
\[\|(\hat{u}- \eps \hat{u}_1) - \hat{u}\|_{0,s} = \eps \|\hat{u}_1\|_{0,s} < r,\]
so $\hat{u}- \eps \hat{u}_1 \in C_s^0(\overline{\Omega})_+$, in particular $\eps \hat{u}_1 \le \hat{u}$. Therefore $(\eps \hat{u}_1, \hat{u})$ is a sub-supersolution pair of \eqref{p}.
\vskip2pt
\noindent
For all $n \in \N$ big enough, $\eps = \frac{1}{n}$ satisfies \eqref{ep}. By Theorem \ref{SG}, there exists
\[u_n = \min \mathcal{S}\Big(\frac{\hat{u}_1}{n}, \hat{u}\Big).\]
Clearly $(0,\hat{u})$ is a sub-supersolution pair of \eqref{p} and $u_n\in\mathcal{S}(0,\hat{u})$, so by Lemma \ref{comp}, passing if necessary to a subsequence, we have $u_n\to u_+$ in $\w$ for some $u_+\in\mathcal{S}(0,\hat{u})$.
\vskip2pt
\noindent
On the other hand we have for all $n\in\N$
\[\mathcal{S}\Big(\frac{\hat{u}_1}{n},\hat{u}\Big) \subseteq \mathcal{S}\Big(\frac{\hat{u}_1}{n+1},\hat{u}\Big),\]
hence by minimality $u_{n+1}\le u_n$. This in turn implies that $u_n(x)\to u_+(x)$ for a.e.\ $x\in\Omega$. Now, since $0\le u_n\le\hat{u}$, we see that $(u_n)$ is a bounded sequence in $L^\infty(\Omega)$, hence by ${\bf H}_1$ \ref{h1} $(f(\cdot,u_n))$ is uniformly bounded as well. Then, since for all $n\in\N$
\beq\label{eqn}
\fpl u_n = f(\cdot,u_n) \text{ in } W^{-s,p'}(\Omega),
\eeq
Lemmas \ref{bound}, \ref{reg} imply that $(u_n)$ is bounded in $C_s^\alpha(\overline\Omega)$ as well. So, passing to a further subsequence, we have $u_n\to u_+$ in $C_s^0(\overline\Omega)$.
\vskip2pt
\noindent
We prove now that $u_+ \neq 0$, by contradiction. If $u_+ = 0$, then $u_n \to 0$ uniformly in $\overline{\Omega}$. Set
\[v_n = \frac{u_n}{\|u_n\|_{s,p}} \in \w_+,\]
then by \eqref{eqn} we have for all $n \in \N$
\[\fpl v_n = \frac{f(\cdot, u_n)}{\|u_n\|_{s,p} ^{p-1}} = \frac{f(\cdot,u_n)}{u_n^{p-1}} v_n^{p-1} \text{ in } W^{-s,p'}(\Omega).\]
Set for all $n\in\N$
\[\rho_n = \frac{f(\cdot,u_n)}{u_n^{p-1}},\]
By \eqref{st}, for $n\in\N$ big enough we have $c_1 \le \rho_n \le c_2$ in $\Omega$, in particular $\rho_n \in L^{\infty}(\Omega)$.
Then $v_n \in \w \setminus \{0\}$ is an eigenfunction of the \eqref{p4}-type eigenvalue problem
\begin{equation}
\fpl v_n = \lambda \rho_n v_n^{p-1} \text{ in } W^{-s,p'}(\Omega),
\label{wep}
\end{equation}
associated with the eigenvalue $\lambda =1$. Since $\rho_n \ge c_1 > \lambda_1$, by Lemma \ref{sp} \ref{sp3} we have
\[\lambda_1 (\rho_n) < \lambda_1 (\lambda_1) = 1,\]
therefore $v_n$ is a non-principal eigenfunction of \eqref{wep}. By Lemma \ref{sp} \ref{sp2} $v_n$ is nodal, a contradiction.
Hence, by Lemma \ref{DQ} and \eqref{int} we have $u_{+} \in \mathrm{int}(C_s^0(\overline{\Omega})_+)$.
\vskip2pt
\noindent
Finally, we prove that $u_+$ is the smallest positive solution of \eqref{p}. Let $u \in \w_+ \setminus \{0\}$ be a solution of \eqref{p}. Arguing as above we see that 
$u \in \mathrm{int}(C_s^0(\overline{\Omega})_+)$. Set $w=u \land \hat{u} \in \w_+$, then by Lemma \ref{Mm} $w$ is a supersolution of \eqref{p}. As above, for all $n \in\N$ big enough we have that $\frac{\hat{u}_1}{n}$ is a subsolution of \eqref{p} and $\frac{\hat{u}_1}{n} \le w$ in $\Omega$, i.e., $(\hat{u}_1/n,w)$ is a sub-supersolution pair. Therefore, by Lemma \ref{Ex} we can find 
\[w_n \in \mathcal{S}\Big(\frac{\hat{u}_1}{n}, w\Big).\]
Since
\[\mathcal{S}\Big(\frac{\hat{u}_1}{n}, w\Big) \subseteq \mathcal{S}\Big(\frac{\hat{u}_1}{n}, \hat{u}\Big),\]
by minimality, for all $n\in\N$ big enough we have $u_n \le w_n$, hence $u_n \le u$. Passing to the limit as $n \to \infty$, we have $u_+ \le u$.
\vskip2pt
\noindent
Similarly we prove existence of the biggest negative solution $u_{-} \in -\mathrm{int}(C_s^0(\overline{\Omega})_+)$.
\end{proof}

\begin{remark}
According to \cite{HS}, most properties on Lemma \ref{sp} also hold if $\rho$ lies in a special class $\widetilde{W}_p$ of {\em singular} weights, namely if $\rho {\rm d}_{\Omega}^{sa} \in L^r(\Omega)$ for some $a \in [0,1]$, $r>1$ satisfying 
\[\frac{1}{r} + \frac{a}{p} + \frac{p-a}{p_s^*} < 1.\]
So, in view of the proof of Theorem \ref{main} above, a natural question is whether we may replace ${\bf H}_1$ \ref{h3} with the weaker condition
\[ \liminf_{t\to 0}\frac{f(x,t)}{t^{p-1}} > \lambda_1 \text{ uniformly for a.e.\ } x\in\Omega.\]
Define $\rho_n=f(\cdot,u_n)/u_n^{p-1}$ as above, then recalling that $u_n\ge c\ds$ in $\overline\Omega$ we have
\[ 0 < \rho_n \le C (1+ d_{\Omega}^{-s(p+1)}).\]
Unfortunately, this {\em does not} ensure that $\rho_n\in\widetilde{W}_p$, in general. For instance, consider the case $\Omega=B_1(0)$, ${\rm d}_\Omega(x)=1-|x|$. Then we have ${\rm d}_\Omega^s\in L^\alpha(\Omega)$ iff $\alpha\in (0,1)$. Therefore, $\rho_n \in \widetilde{W}_p$ implies
\[\begin{cases}
sr(p-a-1) <1 & \text{} \\
\displaystyle\frac{1}{r} + \frac{a}{p} + \frac{p-a}{p_s^*} < 1, & \text{}
\end{cases}\]
in particular $(p-2)s< 1$. Yet, for special values of $p$, $s$, and a suitable domain $\Omega$, analogues to Theorem \ref{main} could be proved for reactions $f(x,\cdot)$ with a $(p-1)$-sublinear behavior near the origin.
\end{remark}

\section{Nodal solutions}\label{sec5}

\noindent
In this section we present an application of our main result, following the ideas of \cite{FP} (see also \cite[Theorem 11.26]{MMP}). Applying Theorem \ref{main}, along with the mountain pass theorem and spectral theory for $\fpl$, we prove existence of a nodal solution of \eqref{p}. Our hypotheses on the reaction $f$ are the following:
\begin{itemize}[leftmargin=1cm]
\item[${\bf H}_2$] $f:\Omega\times\R\to\R$ is a Carath\'{e}odory function, for all $(x,t)\in\Omega\times\R$ we set
\[F(x,t)=\int_0^t  f(x,\tau)\,d\tau,\]
and the following conditions hold:
\begin{enumroman}
\item\label{h_1} $|f(x,t)| \leq c_0(1+|t|^{q-1})$ for all a.e.\ $x\in\Omega$ and all $t\in\R$ ($c_0>0$, $q \in(p,p_s^*)$);
\item\label{h_2} $\displaystyle\limsup_{|t|\to\infty} \frac{F(x,t)}{|t|^p}<\frac{\lambda_1}{p}$ uniformly for a.e.\ $x\in\Omega$;
\item\label{h_3} $\displaystyle\lambda_2 < \liminf_{t\to 0}\frac{f(x,t)}{|t|^{p-2}t} \le \limsup_{t\to 0}\frac{f(x,t)}{|t|^{p-2}t} < \infty$ uniformly for a.e.\ $x\in\Omega$.
\end{enumroman}
\end{itemize}
Here $\lambda_2 > \lambda_1$ denotes the second (variational) eigenvalue of $\fpl$ in $\w$, defined by \eqref{e2}. Again, we are assuming for $f(x,\cdot)$ a $(p-1)$-linear behavior near the origin.
\vskip2pt
\noindent
Our method is variational. We define the energy functional $\Phi$ as in Section \ref{sec1} and recall the following \emph{Palais-Smale compactness condition}:
\begin{enumerate}[label=(PS)]
\item\label{ps} Any sequence $(u_n)_n$ in $\w$, s.t.\ $(\Phi(u_n))$ is bounded in $\R$ and $\Phi'(u_n)\rightarrow 0$ in $W^{-s,p'}$, admits a (strongly) convergent subsequence.
\end{enumerate}
We will use the following notation for critical points:
\[K(\Phi) = \big\{u\in\w:\,\Phi'(u)=0 \ \text{in} \ W^{-s,p'}(\Omega)\big\}\,\]
and for all $c\in\R$
\[K_c(\Phi) = \big\{u\in K(\Phi):\,\Phi(u)=c\big\}.\]
Our result is the following:

\begin{theorem}\label{nod}
Let ${\bf H}_2$ hold. Then, \eqref{p} has a smallest positive solution $u_{+} \in \mathrm{int}(C_s^0(\overline{\Omega})_+)$, a biggest negative solution 
$u_{-} \in -\mathrm{int}(C_s^0(\overline{\Omega})_+)$, and a nodal solution $\tilde{u} \in C_s^0(\overline{\Omega})$ s.t.\ $u_{-} \le \tilde{u} \le u_{+}$ in $\Omega$.
\end{theorem}
\begin{proof}
Clearly ${\bf H}_2$ implies ${\bf H}_1$. From Theorem \ref{main}, then, we know that \eqref{p} has a smallest positive solution $u_{+} \in \mathrm{int}(C_s^0(\overline{\Omega})_+)$ and a biggest negative solution $u_{-} \in -\mathrm{int}(C_s^0(\overline{\Omega})_+)$. Plus, by  ${\bf H}_2$ \ref{h_3}, $0$ is a solution of \eqref{p}. We are going to detect a fourth solution $\tilde u\in\w$, and then show that it is nodal.
\vskip2pt
\noindent
Set for all $(x,t) \in \Omega \times \R$
\[\tilde{f}(x,t)=
\begin{cases}
f(x,u_{-}(x)) \ & \text{ if } t < u_{-}(x) \\
f(x,t) \ & \text{ if } u_{-}(x) \le t \le u_{+}(x) \\
f(x, u_{+}(x)) & \text{ if } t > u_{+}
\end{cases}\]
and
\[\tilde{F}(x,t)= \int_0^t  \tilde{f}(x,\tau)\,d\tau.\]
Since $u_{\pm} \in L^{\infty}(\Omega)$, $\tilde{f}$ satisfies ${\bf H}_0$. Now set for all $u \in \w$
\[\tilde{\Phi}(u)= \frac{\|u\|_{s,p}^p}{p} - \int_{\Omega} \tilde{F}(x,u)\,dx.\]
By ${\bf H}_2$ \ref{h_1} \ref{h_2}, reasoning as in the proof of Theorem \ref{main} we see that $\tilde{\Phi} \in C^1(\w)$ is coercive. As a consequence, $\tilde{\Phi}$ satisfies \ref{ps} (see \cite[Proposition 2.1]{ILPS}). 
Whenever $u \in \w$ is a critical point of $\tilde{\Phi}$, then for all $v\in \w$
\begin{equation}
\langle \fpl u, v \rangle = \int_{\Omega} \tilde{f}(x,u)v\,dx. 
\label{wf}
\end{equation}
By Lemmas \ref{bound}, \ref{reg} we have $u \in C_s^0(\overline{\Omega})$. Besides, testing \eqref{wf} with $(u-u_+)^+, -(u-u_{-})^{-} \in \w$ and arguing as in Lemma \ref{Ex} we have $u_{-} \le u \le u_{+}$ in $\Omega$, hence $u$ solves \eqref{p} in $\Omega$. Using the notation of Section \ref{sec3}, we can say that $u\in\mathcal{S}(u_-,u_+)$.
\vskip2pt
\noindent
We introduce a further truncation setting for all $(x,t) \in \Omega \times \R$
\[\tilde{f}_{+}(x,t)=\tilde{f}(x, t^{+}), \quad \tilde{F}_{+}(x,t)= \int_0^t  \tilde{f}_{+}(x,\tau)\,d\tau,\]
and for all $u\in \w$
\[\tilde{\Phi}_{+}(u)= \frac{\|u\|_{s,p}^p}{p} - \int_{\Omega} \tilde{F}_{+}(x,u)\,dx.\]
Reasoning as above, we see that $\tilde{\Phi}_{+} \in C^1(\w)$ is coercive, and whenever $u\in\w$ is a critical point of $\tilde{\Phi}_+$ we have $u\in\mathcal{S}(0,u_+)$. By the compact embedding $\w\hookrightarrow L^q(\Omega)$, it is easily seen that $\tilde{\Phi}_+$ is sequentially weakly lower semicontinuous, hence there exists $\tilde{u}_{+} \in \w$ s.t.\  
\[\tilde{\Phi}_{+}(\tilde{u}_{+})=\inf_{u \in \w} \tilde{\Phi}_{+}(u).\]
Arguing as in Theorem \ref{main} we see that $\tilde{\Phi}_{+}(\tilde{u}_{+})<0$, hence $\tilde{u}_{+} \neq 0$. By ${\bf H}_2$ \ref{h_3} and Lemma \ref{DQ}, we have $\tilde{u}_{+} \in \mathrm{int}(C_s^0(\overline{\Omega})_+)$. So, $\tilde{u}_{+}$ is a positive solution of \eqref{p}, hence the minimality of $u_{+}$ implies $\tilde{u}_{+} =u_{+}$. In particular, since $\tilde{\Phi}=\tilde{\Phi}_+$ in $C_s^0(\overline\Omega)_+$, we see that $u_{+} \in \mathrm{int}(C_s^0(\overline{\Omega})_+)$ is a local minimizer of $\tilde{\Phi}$ in $C_s^0(\overline{\Omega})$. By Lemma \ref{svh}, then $u_+$ is a local minimizer of $\tilde{\Phi}$ in $\w$ as well (recall that $\tilde f$ sarisfies ${\bf H}_0$).
\vskip2pt
\noindent
Similarly we prove that $u_{-} \in -\mathrm{int}(C_s^0(\overline{\Omega})_+)$ is a local minimizer of $\tilde{\Phi}$.
\vskip2pt
\noindent
Now we argue by contradiction, assuming that there are no other critical points of $\tilde{\Phi}$ than $0$, $u_+$, and $u_-$, namely,
\begin{equation}
K(\tilde{\Phi})=\{0, u_{+}, u_{-}\}.
\label{cp}
\end{equation}
In particular, both $u_{\pm}$ are strict local minimizers of $\tilde{\Phi}$, which satisfies \ref{ps}. By the mountain pass Theorem \cite[Proposition 5.42]{MMP}, there exists $\tilde{u} \in K_c(\tilde{\Phi})$, where we have set
\[\Gamma=\big\{\gamma \in C([0,1],\w): \gamma(0)=u_{+}, \gamma(1)=u_{-}\big\},\]
and
\[c=\inf_{\gamma \in \Gamma} \max_{t \in [0,1]} \tilde{\Phi}(\gamma(t)) > \max\big\{\tilde{\Phi}(u_{+}), \tilde{\Phi}(u_{-})\big\}.\]
In particular $\tilde{u}\neq u_\pm$, which by \eqref{cp} implies $\tilde{u}=0$ and hence $c=0$. Set 
\[\Sigma=\{u \in \w \cup C_s^0(\overline{\Omega}): \|u\|_p=1\}.\]
By ${\bf H}_2$ \ref{h_3} we can find $\mu > \lambda_2$, $\delta >0$ s.t.\ for all $x \in \Omega$, $|t| \le \delta$
\[F(x,t) \ge \frac{\mu}{p} |t|^p.\]
By \eqref{e2} there is $\gamma_1 \in \Gamma_1$ s.t.\
\[\max_{t \in [0,1]} \|\gamma_1(t)\|_{s,p}^p < \mu,\]
and by density we may assume $\gamma_1 \in C([0,1], \Sigma)$, continuous with respect to the $C_s^0(\overline{\Omega})$-norm (see \cite{DI} for details). Since $t\mapsto\|\gamma_1(t)\|_\infty$ is bounded in $[0,1]$, we can find $\eps>0$ s.t.\ $\|\eps\gamma_1(t)\|_\infty\le\delta$ for all $t\in[0,1]$.
\vskip2pt
\noindent
Besides, taking $\eps>0$ even smaller if necessary, we have for all $t\in[0,1]$
\[u_{+} - \eps_t \gamma_1(t) \in \mathrm{int}(C_s^0(\overline{\Omega})_+), \quad u_{-} - \eps_t \gamma_1(t) \in -\mathrm{int}(C_s^0(\overline{\Omega})_+),\]
in particular $u_{-} < \eps \gamma_1(t) < u_{+}$ a.e.\ in $\Omega$. So, for all $t\in[0,1]$ we get
\begin{align*}
\tilde{\Phi}(\eps \gamma_1(t)) &=  \frac{\eps^p}{p} \|\gamma_1(t)\|_{s,p}^p - \int_{\Omega} \tilde{F}(x,\eps \gamma_1(t))\,dx \\
&\le \frac{\eps^p}{p} \|\gamma_1(t)\|_{s,p}^p - \frac{\mu \eps^p}{p} \|\gamma_1(t)\|_{p}^p \\
&= \frac{\eps^p}{p} (\|\gamma_1(t)\|_{s,p}^p- \mu) < 0.
\end{align*}
Thus, $\eps\gamma_1$ is a continuous path joining $\eps\hat{u}_1$ to $-\eps\hat{u}_1$, s.t.\ for all $t\in[0,1]$
\[\tilde{\Phi}(\eps\gamma_1(t)) < 0.\]
Besides, by \eqref{cp} and Lemma \ref{DQ} we have 
\[K(\tilde{\Phi}_{+})=\{0, u_{+}\},\]
Set $a=\tilde{\Phi}_{+}(u_+)$, $b=\tilde{\Phi}_{+}(\eps \hat{u}_1)$, hence $a<b<0$ and there is no critical level in $(a,b]$. Therefore, by the second deformation theorem \cite[Theorem 5.34]{MMP} there exists a continuous deformation $h:[0,1] \times \{\tilde{\Phi}_{+} \le b\} \rightarrow \{\tilde{\Phi}_{+} \le b\}$ s.t.\ for all $t \in [0,1]$, $\tilde{\Phi}_{+}(u) \le b$
\[h(0,u)=u, \quad h(1,u)=u_{+}, \quad \tilde{\Phi}_{+}(h(t,u)) \le \tilde{\Phi}_{+}(u).\]
Set for all $t \in [0,1]$
\[\gamma_{+}(t)=h(t,\eps \hat{u}_1)^{+} \in \w_{+},\]
then $\gamma_{+} \in C([0,1], \w)$ with $\gamma_{+}(0)=\eps \hat{u}_1$, $\gamma(1)=u_{+}$, and for all $t \in [0,1]$
\[\tilde{\Phi}(\gamma_{+}(t)) \le b < 0.\]
Similarly we construct $\gamma_{-} \in C([0,1], \w)$ s.t.\ $\gamma_{-}(0)=-\eps \hat{u}_1$, $\gamma(1)=u_{-}$, and for all $t \in [0,1]$
\[\tilde{\Phi}(\gamma_{-}(t)) < 0.\]
Concatenating $\gamma_{+}, \eps \gamma_1, \gamma_{-}$ we find a path $\gamma \in \Gamma$ s.t.\ for all $t \in [0,1]$
\[\tilde{\Phi}(\gamma(t)) < 0,\]
hence $c<0$, a contradiction. So, \eqref{cp} is false, i.e., there exists $\tilde{u}\in K(\tilde\Phi)\setminus\{0,u_+,u_-\}$, so as ween above we have $\tilde{u}\in\mathcal{S}(u_-,u_+)$.
\vskip2pt
\noindent
Finally, we prove that $\tilde{u}$ is nodal. Indeed, if $\tilde{u}\in\w_+\setminus\{0\}$, then by Lemma \ref{DQ} we would have $\tilde{u}\in\mathrm{int}(C_s^0(\overline{\Omega})_+)$, along with $\tilde{u}\le u_+$, which, by Theorem \ref{main}, would imply $\tilde{u}=u_+$, a contradiction. Similarly we see that $\tilde{u}$ cannot be negative.
\vskip2pt
\noindent
Thus, $\tilde{u}\in C_s^0(\overline\Omega)\setminus\{0\}$ is a nodal solution of \eqref{p} s.t.\ $u_-\le\tilde{u}\le u_+$ a.e.\ in $\Omega$.
\end{proof}

\begin{remark}
The argument based on the characterization of $\lambda_2$ was already employed in \cite[Theorem 4.1]{IMS3} and \cite[Theorem 3.3]{DI} (for $p=2$). The novelty of Theorem \ref{nod} above, with respect to such results (even for the linear case $p=2$), lies in the detailed information about solutions, as we prove that $u_\pm$ are {\em extremal} constant sign solutions and $\tilde{u}$ is {\em nodal}. We also remark that the assumption $p\ge 2$ is essentially due to regularity theory (Lemma \ref{reg}), but the arguments displayed in this paper also work, with minor adjustments, for $p\in(1,2)$.
\end{remark}

\vskip4pt
\noindent
{\bf Acknowledgement.} Both authors are members of GNAMPA (Gruppo Nazionale per l'Analisi Matematica, la Probabilit\`a e le loro Applicazioni) of INdAM (Istituto Nazionale di Alta Matematica 'Francesco Severi'). A.\ Iannizzotto is supported by the grant PRIN n.\ 2017AYM8XW: {\em Non-linear Differential Problems via Variational, Topological and Set-valued Methods}, and by the research project {\em Integro-differential Equations and nonlocal Problems} funded by Fondazione di Sardegna (2017). We thank S.\ Mosconi for useful discussions.

\end{document}